\documentclass[oneside,english]{amsart}
\usepackage[T1]{fontenc}
\usepackage[latin9]{inputenc}
\usepackage{amsthm}
\usepackage{amstext}
\usepackage{amssymb}
\usepackage{esint}
\usepackage{graphicx}
\usepackage{color}

\makeatletter
\numberwithin{equation}{section}
\numberwithin{figure}{section}
\theoremstyle{plain}
\newtheorem{thm}{\protect\theoremname}[section]
  
  \theoremstyle{plain}
  \theoremstyle{definition}
  
  \theoremstyle{plain}
  \newtheorem{lem}[thm]{\protect\lemmaname}
  \newtheorem{cor}[thm]{\protect\corollaryname}
  \theoremstyle{plain}
	\newtheorem{rem}[thm]{\protect\remarkname}
  \theoremstyle{plain}
	\theoremstyle{plain}
	
\makeatother

\usepackage{babel}
  \providecommand{\definitionname}{Definition}
  \providecommand{\lemmaname}{Lemma}
  \providecommand{\theoremname}{Theorem}
  \providecommand{\corollaryname}{Corollary}
  \providecommand{\remarkname}{Remark}
  \providecommand{\propositionname}{Proposition}
  \providecommand{\examplename}{Example}

	\DeclareMathOperator{\loc}{loc}
	\DeclareMathOperator{\adj}{adj}

	\DeclareMathOperator{\cp}{cap}

\begin{document}

\title[Composition operators on Sobolev spaces and  Ball's classes]{Composition operators on Sobolev spaces and  Ball's classes}

\author{V.~Gol'dshtein, A.~Ukhlov} 

\begin{abstract}
In this paper we study geometric aspects of  Ball's classes in the context of nonlinear elasticity problems. The suggested approach is based on the characterization of Ball's classes $A_{q,q'}(\Omega)$ in terms of composition operators on Sobolev spaces. This characterization allows us to obtain the volume compression (distortion) estimates of topological mappings of Ball's classes. We prove also that Sobolev homeomorphisms of Ball's classes which possess the Luzin  $N$-property are absolutely continuous with respect to capacity.
\end{abstract}
\maketitle
\footnotetext{\textbf{Data availability:} Data sharing not applicable to this article as no datasets were generated or analysed during the current study.}

\section{Introduction}

\subsection{Ball's classes}
Let $\Omega$ be a bounded domain in $\mathbb R^n$, $n\geq 3$, and $\varphi: \Omega \to \widetilde{\Omega}$ be a measurable mapping. In fundamental papers \cite{B76,B81} J.~M.~Ball demonstrated that solutions of typical boundary value nonlinear elasticity problems are minimizers of complicated energy integrals of deformations 
$$
I(\varphi,\Omega):=\int_{\Omega}W(x,\varphi(x), D\varphi(x))~dx,
$$
where  $\varphi: \Omega \to \widetilde{\Omega}$ is a continuous weak differentiable mapping (deformation) that is a homeomorphism of boundaries $\partial \Omega$ and $\partial \widetilde{\Omega}$. 

In the spatial case ($n=3$)  J.~M.~Ball \cite{B76,B81} proved that appropriate classes of minimizers (Ball's classes) are
$$
A^{+}_{q,r}(\Omega;\widetilde{\Omega})=\left\{\varphi\in L^1_q(\Omega; \widetilde{\Omega}): \adj D\varphi \in L_r(\Omega),\,\,J(x,\varphi)>0 \,\,\text{for almost all}\,\,x\in\Omega \right\},
$$
where $q > n-1$ and $r\geq {q}/{(q-1)}$. In the case $q>n$ J.~M.~Ball proved \cite{B76,B81} (with some regularity boundary conditions of the variational problem) existence of the inverse mappings and its weak regularity, namely $\varphi^{-1}$ belongs to the class $L^1_1(\widetilde{\Omega})$ where $\widetilde{\Omega}=\varphi(\Omega)$. 
 
The case $n-1<q\leq n$ was studied in subsequent works \cite{HK04,Sv88} (see, also \cite{MST,MTY}) where it was proved that outside of a  "singular" set $S$ of a variational $q$-capacity zero the mapping $\varphi: \Omega\setminus S \to \widetilde{\Omega}$ is a homeomorphism. Its inverse homeomorphism $\varphi^{-1}$ belongs to the space $L^1_1(\widetilde{\Omega})$ and maps the set $\varphi(\Omega \setminus S)$ onto  $\Omega\setminus S$. It was proved also (\cite{HK04}) that the "singular" set $H:=\widetilde{\Omega} \setminus  \varphi(\Omega \setminus S)$ of $\varphi^{-1}$ has $(n-1)$-Hausdorff measure zero. 

In the present work, using the geometric theory of composition operators on Sobolev spaces $L^1_q$ \cite{U93,VU98,VU04} we introduce a characterization of Ball's classes in terms of the composition operators. As one of our results we obtain accurate volume distortion estimates of topological mappings (homeomorphisms) of Ball's classes possessing the Luzin $N$-property. The absolute continuity of these mappings with respect to an appropriate capacity is proved also. The capacity considered as an outer measure associated with corresponding Sobolev spaces $L^1_q$ \cite{Ch54}. This approach allows us to refine in the capacitory terms results of \cite{HK04,Sv88} about "singular" sets.   

\begin{rem}
In \cite{Sv88} the change of variables formulas in Theorem~2 of \cite{Sv88} and Theorem~3 of \cite{Sv88} are valid if the topological mappings $\varphi$ possess the Luzin $N$-property, since these formulas imply  that for any  set of measure zero ($|A|=0$)
we obtain
$$
|\varphi(A)|=\int\limits_A |J(x,\varphi)|~dx=0.
$$
Thereby, in the case $n-1<q<n$, the assumption about the Luzin N-property was missed in \cite{Sv88}.
\end{rem} 

\subsection{Ball's classes and composition operators}
The classical Ball's classes are two-indexed classes $A^{+}_{q,r}(\Omega;\widetilde{\Omega})$. Recall that in domains of finite measure $\Omega\subset\mathbb R^n$ there is an embedding $L_{r_1}(\Omega)\subset L_{r_1}(\Omega)$ if $r_1\leq r_2$, and so
$$
A^{+}_{q,r_1}(\Omega;\widetilde{\Omega})\subset A^{+}_{q,r_2}(\Omega;\widetilde{\Omega})\,\,\text{if}\,\,r_1\leq r_2.
$$
In the present work we study Ball's classes $A^{+}_{q,r}(\Omega;\widetilde{\Omega})$ in the case of the minimal possible second index $r$, namely $r=q/(q-1)$ or $1/q+1/r=1$ and so $r=q'$. We consider functional and geometric properties of one-indexed classes 
$A^{+}_{q,q'}(\Omega;\widetilde{\Omega})$ which are Sobolev homeomorphisms $\varphi:\Omega\to\widetilde{\Omega}$ of the class $L^1_q(\Omega;\widetilde{\Omega})$, $q>n-1$, such that  $\adj D\varphi \in L_{q'}(\Omega)$, $1/q+1/q'=1$, $J(x,\varphi)>0$ for almost all $x\in\Omega$ and $\varphi$ possesses the Luzin $N$-property.

\begin{rem}
In the case $q=n$ ($q'=n/(n-1)$) by the duality composition theorem (Theorem~\ref{CompDualLim}) the Ball's class $A^{+}_{n,n'}(\Omega;\widetilde{\Omega})$ coincides with the Sobolev space $L^1_n(\Omega;\widetilde{\Omega})$ of mappings with $J(x,\varphi)>0$ for almost all $x\in\Omega$.
\end{rem}

The class of Sobolev homeomorphisms $L^1_n(\Omega;\widetilde{\Omega})$ was studied in details in \cite{S85} under a name $BL$-homeo\-mor\-phisms.

\begin{rem}
In \cite{VGR79} it was proved that Sobolev homeomorphisms $\varphi\in L^1_{n,\loc}(\Omega;\widetilde{\Omega})$ possess the Luzin $N$-property. Therefore Ball's classes   $A^{+}_{q,q'}(\Omega;\widetilde{\Omega})$, $n\leq q$ have the Luzin $N$-property. 
\end{rem}

We use the terminology  Sobolev homeomorphism because a possible existence of "singular sets" in $\Omega$ \cite{Sv88}
for the case $n-1<q<n$. It will be demonstrated later that the "singular" sets are sets of capacity zero.

The proposed approach to the geometric properties of Ball's classes $A^{+}_{q,q'}(\Omega;\widetilde{\Omega})$ is based on the geometric theory of composition operators on Sobolev spaces \cite{GGR95,GU10,U93,VU98,VU02,VU04,VU05}. We show that a Sobolev topological mapping (homeomorphism) $\varphi: \Omega \to \widetilde{\Omega}$ such that $J(x,\varphi)>0$ for almost all $x\in\Omega$, belongs to $A^{+}_{q,q'}(\Omega;\widetilde{\Omega})$ if and only if for any function $f\in L^1_{\infty}(\widetilde{\Omega})$ (i.e. for bounded gradients of the stress tensor after the deformation ), the composition $\varphi^{\ast}(f)=f\circ\varphi\in  L^1_q(\Omega)$ and the following inequality is correct
$$
\|f\mid L^1_1(\widetilde{\Omega})\|\leq \|\adj D\varphi | L_{q'}(\Omega)\|\cdot\|\varphi^{\ast}(f)\mid L^1_q(\Omega)\|,
$$
where $1/q+1/q'=1$. This inequality states that the second Ball's condition $\adj D\varphi\in L_{q'}(\Omega)$
is equivalent (for mappings with positive a.~e. Jacobians) to the boundedness of the composition operator 
\begin{equation}
\label{eq:1.0}
\left(\varphi^{-1}\right)^{\ast}: L^1_{q}(\Omega)\to L^1_1(\widetilde{\Omega}),\,\,n-1<q<\infty,
\end{equation}
generated by the inverse mapping (inverse deformation) $\varphi^{-1}:\widetilde{\Omega}\to\Omega$. 

Using this composition operators property (\ref{eq:1.0})  and results of papers \cite{VU98,VU02,VU04} we obtain the following volume distortion estimates for mappings of Ball's classes (Theorem~\ref{distmes} ): 

\noindent
{\it Let $\varphi:\Omega \to \widetilde{\Omega}$ belongs to $A^{+}_{q,q'}(\Omega;\widetilde{\Omega})$, then the following inequality
$$
|\varphi(A)|^{1-\frac{1}{n}}\leq \left(\frac{1}{|A|}\int\limits_A|\adj D\varphi(x)|^{\frac{q}{q-1}}~dx\right)^{1-\frac{1}{q}} \,|A|^{1-\frac{1}{n}}, \quad n-1< q\leq n,
$$
holds for any measurable set $A\subset\Omega$. }

Using capacitary properties of composition operators on Sobolev spaces we obtain corresponding estimates for capacity distortion  of mappings of the Ball's classes $A^{+}_{q,r}(\Omega;\widetilde{\Omega})$ (Theorem~\ref{distcap}):

\noindent
{\it Let $\varphi:\Omega \to \widetilde{\Omega}$ belongs to $A^{+}_{q,q'}(\Omega;\widetilde{\Omega})$, then the following capacity inequality
$$
\cp_{1}(\varphi(E);\widetilde{\Omega})
\leq \|\adj D\varphi | L_{q/(q-1)}(\Omega)\|\cp_{q}^{\frac{1}{q}}(E;{\Omega}), \quad n-1< q<\infty,
$$
holds for any Borel set $E\subset\Omega$. Corresponding local and point wise estimates are also correct.}

This means that topological mappings of Ball's classes  are absolutely continuous with respect to capacity and an image of cavitation of $q$-capacity zero ("singular" set) in a body $\Omega$ has $1$-capacity zero (also "singular" set) in $\widetilde{\Omega}$ and can not lead to the body breaking upon these deformations. So, the cavitation can be characterized in capacitary terms. As a consequence we obtain characterization of cavitations in terms of the Hausdorff measure: for every set $E\subset{\Omega}$ such that $\cp_q(E;\Omega)=0$, $1\leq q\leq n$, the Hausdorff measure $H^{n-1}(\varphi(E))=0$.

\subsection{Composition operators on Sobolev spaces}

The geometric theory of composition operators on Sobolev spaces is based on the nonlinear potential theory \cite{HKM,M} and  in accordance to this approach  we consider Sobolev homeomorphisms of the class $L^1_q(\Omega)$ defined up to a ("singular") set of $q$-capacity zero. 
The composition operators on Sobolev spaces have the characterization in terms of weak ($p,q$)-quasiconformal mappings. Recall that a homeomorphism $\varphi:\Omega\to\widetilde{\Omega}$ of bounded domains $\Omega,\widetilde{\Omega}\subset\mathbb R^n$, is called a weak $(p,q)$-quasiconformal mapping, $1\leq q\leq p<\infty$, \cite{VU98}, if $\varphi$ belongs to the Sobolev space $L^1_{q}(\Omega;\widetilde{\Omega})$, has finite distortion, and 
\[
K_{p,q}(\varphi;\Omega)=\|K_p \mid L_{\kappa}(\Omega)\|<\infty,\,\,\,1/\kappa=1/q-1/p\,\,(\kappa=\infty \,\,\text{if}\,\,p=q),
\]
where the $p$-dilatation of the mapping $\varphi$ at a point $x$ is defined by
$$
K_p(x)=\inf \{k(x): |D\varphi(x)|\leq k(x) |J(x,\varphi)|^{\frac{1}{p}},\,\,x\in\Omega \}.
$$

\begin{rem}
A weakly differentiable mapping $\varphi:\Omega\to\mathbb{R}^{n}$ is the mapping of finite distortion if $D\varphi(x)=0$ for almost all $x$ from $Z=\{x\in\Omega: J(x,\varphi)=0\}$ \cite{VGR79}. 
Mappings of the Ball's classes are mappings of finite distortion due to the property 
$J(x,\varphi)>0$ for almost all  $x\in\Omega$.
\end{rem}

In \cite{U93,VU02} in the framework of the geometric analysis of Sobolev spaces it was proved that $\varphi:\Omega\to\widetilde{\Omega}$ generates by the chain rule $\varphi^{\ast}(f)=f\circ\varphi$ a bounded operator \begin{equation}
\label{eq:1}
\varphi^{\ast}: L^1_p(\widetilde{\Omega})\to L^1_q(\Omega),\,\,1\leq q \leq p < \infty.
\end{equation}
if and only if $\varphi$ is a weak $(p,q)$-quasiconformal mapping. The limit case $p=\infty$ was considered in \cite{GU10}. It was proved that $\varphi\in L^1_q(\Omega)$ 
if and only if the composition operator
\begin{equation}
\label{eq:2}
\varphi^{\ast}: L^1_{\infty}(\widetilde{\Omega})\to L^1_q(\Omega),\,\,1\leq q <\infty,
\end{equation}
is bounded.

{\it Let us consider in more details the case $n=3$ and $q=n-1=2$. Then the condition (\ref{eq:1.0}) is equivalent 
$$
\left(\int\limits_{\widetilde{\Omega}}\frac{|D\varphi^{-1}(y)|^2}{|J(y,\varphi^{-1})|}~dy\right)^{\frac{1}{2}}<\infty.
$$
It means that a ratio of deformations of line elements to deformations of the surface elements of codimension $1$ (of the admissible inverse deformation) must be integrable.  This looks (more or less) natural in a spirit of nonlinear elasticity theory.}

In the present work we give the natural characterization of Ball's classes in terms of the composition operators. Since the composition operators have geometric descriptions in terms of integral dilatations our approach can be interpreted as a geometric description of deformations of elastic bodies.

\begin{rem}
{\bf The functional definition of Ball's classes.}
We can define the Ball classes $A^{+}_{q,q'}(\Omega;\widetilde{\Omega})$, $q>n-1$, $1/q+1/q'=1$, as a class of Sobolev mappings $\varphi:\Omega\to\widetilde{\Omega}$, having the Luzin $N$-property, with $J(x,\varphi)>0$ for almost all $x\in\Omega$ and such that 
$$
\|f\mid L^1_1(\widetilde{\Omega})\|\leq \|\adj D\varphi | L_{q'}(\Omega)\|\cdot\|\varphi^{\ast}(f)\mid L^1_q(\Omega)\|,
$$
for any $f\in L^1_{\infty}(\widetilde{\Omega})$.
\end{rem}

\noindent
{\bf Short historical remarks.} Weak regularity properties of mappings inverse to Sobolev homeomorphisms were intensively studied in the last decades. In frameworks of the geometric theory of composition operators on Sobolev spaces in \cite{U04} it was proved that mappings inverse to Sobolev homeomorphisms of the class $L^1_q(\Omega)$, $q>n-1$, belong to the class $BV(\widetilde{\Omega})$ and in \cite{GU10} it was proved that mappings inverse to Sobolev homeomorphisms of finite distortions of the class $L^1_{n-1}(\Omega)$ possess the Luzin $N$-property, belong to the class $L^1_1(\widetilde{\Omega})$ and generate a bounded composition operator
$$
\left(\varphi^{-1}\right)^{\ast}: L^1_{\infty}(\Omega)\to L^1_1(\widetilde{\Omega}).
$$
In \cite{CHM} was studied weak differentiability of mappings which are inverse to Sobolev homeomorphisms of finite distortions of the class $W^1_{n-1,\loc}(\Omega)$. Unfortunately the proof of the main theorem in \cite{CHM} contains a serious gap: it was claimed  that the norm of an weak inverse differential can be redefined as equal to zero on a singular set $S$ of measure zero. But the weak inverse differential can be redefined arbitrary only on the set  $S$ of measure zero (for example redefined by the unity). It means that the proof in the work \cite{CHM} is correct only under an additional assumption about the Luzin $N$-property of  $\varphi$. Therefore the main theorem in \cite{CHM}  only repeats results of \cite{GU10}.

The weak differentiability of mappings inverse to Sobolev homeomorphisms of finite distortion with $|D\varphi|\in L^{n-1,1}(\Omega)$ was proved in \cite{HKM06}. 

In the paper \cite{IOZ19}  were considered bi-Sobolev mappings $\varphi:\Omega\to\widetilde{\Omega}$ such that $\varphi\in L^1_n(\Omega)$ and $\varphi^{-1}\in L^1_n(\widetilde{\Omega})$ in connections with the the nonlinear elasticity problems. This class of mapping was introduced and studied n the middle of the 20th century in \cite{OS65}. Note that these classes coincide with  classes of weak $(n,1)$-quasiconformal mappings \cite{U93,VU02}.

\section{Composition operators on Sobolev spaces}

In this section we give necessary elements of the Sobolev spaces theory and
the geometric theory of composition operators on Sobolev spaces.

\subsection{Sobolev spaces}

Let $E$ be a measurable subset of $\mathbb R^n$, $n\geq 2$. The Lebesgue space $L_p(E)$, $1\leq p\leq\infty$, is defined as a Banach space of $p$-summable functions $f:E\to \mathbb R$ equipped with the standard norm.
Lebesgue spaces are Banach spaces of equivalence classes. Two functions belongs to the same equivalence class if they are different on a set of measure zero. Define the Lebesgue value of Lebesgue integrable function as
$$
\tilde{f}(x)=\lim\limits_{r\to 0}\frac{1}{|B(x,r)|}\int\limits_{B(x,r)} f(y)~dy
$$
Lebesgue differentiation theorem states that this limit exists almost everywhere, i.e. the redefined by its Lebesgue value function exist almost everywhere. 

If $\Omega$ is an open subset of $\mathbb R^n$, the Sobolev space $W^1_p(\Omega)$, $1\leq p\leq\infty$, is defined 
as a Banach space of locally integrable weakly differentiable functions
$f:\Omega\to\mathbb{R}$ equipped with the following norm: 
\[
\|f\mid W^1_p(\Omega)\|=\| f\mid L_p(\Omega)\|+\|\nabla f\mid L_p(\Omega)\|.
\]
The homogeneous seminormed Sobolev space $L^1_p(\Omega)$, $1\leq p\leq\infty$, is defined as a space
of locally integrable weakly differentiable functions $f:\Omega\to\mathbb{R}$ equipped
with the following seminorm: 
\[
\|f\mid L^1_p(\Omega)\|=\|\nabla f\mid L_p(\Omega)\|.
\]
Recall that in Lipschitz domains $\Omega\subset\mathbb R^n$, $n\geq 2$, Sobolev spaces $W^1_p(\Omega)$ and $L^1_p(\Omega)$ are coincide (see, for example, \cite{M}).

Sobolev spaces are Banach spaces of equivalence classes \cite{M}. To clarify the notion of equivalence classes we use the nonlinear $p$-capacity associated with Sobolev spaces. With the help of the $p$-capacity the Lebesgue differentiation theorem was refined for Sobolev spaces \cite{M}. 

Recall the definition of the $p$-capacity \cite{GResh,HKM,M}.
Suppose $\Omega$ is an open set in $\mathbb R^n$ and  $F\subset\Omega$ is a compact set. The $p$-capacity of $F$ with respect to $\Omega$ is defined by
\begin{equation*}
\cp_p(F;\Omega) =\inf\{\|\nabla f|L_p(\Omega)\|^p\},
\end{equation*} 
where the infimum is taken over all functions $f\in C_0(\Omega)\cap L^1_p(\Omega)$ such that $f\geq 1$ on $F$.  such functions are called admissible functions for the compact set $F\subset\Omega$. 
If 
$U \subset\Omega$ 
is an open set, we define
\begin{equation*}
\cp_{p}(U;\Omega)=\sup_F\{\cp_{p}
(F;\Omega)\,:\,F\subset U,\,\, F\,\,\text{is compact}\}.
\end{equation*}

In the case of an arbitrary set 
$E\subset\Omega$
we define the inner $p$-capacity 
\begin{equation*}
\underline{\cp}_{p}(E;\Omega)=\sup_f\{\cp_{p}(F;\Omega)\, :\,\,F\subset E\subset\Omega,\,\, F\,\,\text{is compact}\},
\end{equation*}
and the outer $p$-capacity 
\begin{equation*}
\overline{\cp}_{p}(E;\Omega)=\inf_E\{\cp_{p}(U;\Omega)\, :\,\,E\subset U\subset\Omega,\,\, U\,\,\text{is open}\}.
\end{equation*}

A set $E\subset\Omega$ is called $p$-capacity measurable, if $\underline{\cp}_p(E;\Omega)=\overline{\cp}_p(E;\Omega)$. Let $E\subset\Omega$ be a $p$-capacity measurable set. The value
$$
\cp_p(E;\Omega)=\underline{\cp}_p(E;\Omega)=\overline{\cp}_p(E;\Omega)
$$
is called the $p$-capacity measure of the set $E\subset\Omega$. By \cite{Ch54} analytical sets (in particular Borel sets) are $p$-capacity measurable sets.

The notion of the $p$-capacity measure permits us to refine the notion of Sobolev functions. Let a function $f\in L^1_p(\Omega)$. Then the refined function 
$$
\tilde{f}(x)=\lim\limits_{r\to 0}\frac{1}{|B(x,r)|}\int\limits_{B(x,r)} f(y)~dy
$$
is defined quasieverywhere i.~e. up to a set of $p$-capacity zero and it is absolutely continuous on almost all lines \cite{M}. This refined function $\tilde{f}\in L^1_p(\Omega)$ is called the unique quasicontinuous representation ({\it a canonical representation}) of the function $f\in L^1_p(\Omega)$. Recall that a function $\tilde{f}$ is termed quasicontinuous if for any $\varepsilon >0$ there is an open  set $U_{\varepsilon}$ such that the $p$-capacity of $U_{\varepsilon}$ is less than $\varepsilon$ and on the set $\Omega\setminus U_{\varepsilon}$ the function  $\tilde{f}$ is continuous (see, for example \cite{HKM,M}). 
In what follows we will use the quasicontinuous (refined) functions only.

Note that the first weak derivatives of
the function $f$ coincide almost everywhere with the usual
partial derivatives (see, e.g., \cite{M} ).

\subsection{Composition operators on Sobolev spaces}

Let us recall basic notations and results of the geometric theory of composition operators on Sobolev spaces. Suppose $\varphi:\Omega\to\mathbb{R}^{n}$ is a mapping of the Sobolev class $W^1_{1,\loc}(\Omega;\mathbb R^n)$. Then the formal Jacobi
matrix $D\varphi(x)$ and its determinant (Jacobian) $J(x,\varphi)$
are well defined at almost all points $x\in\Omega$. The norm $|D\varphi(x)|$ is the operator norm of $D\varphi(x)$,
i.~e.,  $|D\varphi(x)|=\max\{|D\varphi(x)\cdot h| : h\in\mathbb R^n, |h|=1\}$. We also let $l(D\varphi(x))=\min\{|D\varphi(x)\cdot h| : h\in\mathbb R^n, |h|=1\}$. 

The Sobolev mapping $\varphi:\Omega\to\mathbb{R}^{n}$ is a mapping of finite distortion if $D\varphi(x)=0$ for almost all $x$ from $Z=\{x\in\Omega: J(x,\varphi)=0\}$ \cite{VGR79}. 
Of course the condition $J(x,\varphi)>0$ for almost all $x\in\Omega$ of Ball's classes is stronger then condition of finite distortion, i.e. any mapping of such classes has the finite distortion. It means that all general results of this paper are correct for Ball's classes $A_{q,q'}^{+}(\Omega)$.

Let us reproduce here the change of variable formula for the Lebesgue integral \cite{F69, H93}.
Let a mapping $\varphi : \Omega\to \mathbb R^n$ be such that
there exists a collection of closed sets $\{A_k\}_1^{\infty}$, $A_k\subset A_{k+1}\subset \Omega$ for which restrictions $\varphi \vert_{A_k}$ are Lipschitz mapping on the sets $A_k$ and 
$$
\biggl|\Omega\setminus\sum\limits_{k=1}^{\infty}A_k\biggr|=0.
$$
Then there exists a measurable set $S\subset \Omega$, $|S|=0$ such that  the mapping $\varphi:\Omega\setminus S \to \mathbb R^n$ has the Luzin $N$-property and the change of variable formula
\begin{equation}
\label{chvf}
\int\limits_E f\circ\varphi (x) |J(x,\varphi)|~dx=\int\limits_{\mathbb R^n\setminus \varphi(S)} f(y)N_f(E,y)~dy
\end{equation}
holds for every measurable set $E\subset \Omega$ and every non-negative measurable function $f: \mathbb R^n\to\mathbb R$. Here 
$N_f(y,E)$ is the multiplicity function defined as the number of preimages of $y$ under $f$ in $E$.

Sobolev mappings of the class $W^1_{1,\loc}(\Omega;\mathbb R^n)$ satisfy the conditions of the change of variable formula \cite{H93} and so for Sobolev mappings the change of variable formula \eqref{chvf} holds.

If the mapping $\varphi$ possesses the Luzin $N$-property (the image of a set of measure zero has measure zero), then $|\varphi (S)|=0$ and the second integral can be rewritten as the integral on $\mathbb R^n$. Note, that Sobolev homeomorphisms of the class $W^1_{n,\loc}(\Omega;\widetilde{\Omega})$ possess the Luzin $N$-property \cite{VGR79}.

Let $\Omega$ and $\widetilde{\Omega}$ be bounded domains in $\mathbb R^n$, $n\geq 2$. We say that
a homeomorphism $\varphi:\Omega\to\widetilde{\Omega}$ induces a bounded composition
operator 
\[
\varphi^{\ast}:L^1_p(\widetilde{\Omega})\to L^1_q(\Omega),\,\,\,1\leq q\leq p\leq\infty,
\]
by the composition rule $\varphi^{\ast}(f)=f\circ\varphi$, if for
any function $f\in L^1_p(\widetilde{\Omega})$, the composition $\varphi^{\ast}(f)\in L^1_q(\Omega)$
is defined quasi-everywhere in $\Omega$ and there exists a constant $K_{p,q}(\Omega)<\infty$ such that 
\[
\|\varphi^{\ast}(f)\mid L^1_q(\Omega)\|\leq K_{p,q}(\Omega)\|f\mid L^1_p(\widetilde{\Omega})\|.
\]

\begin{rem}
Sobolev homeomorphisms of the class $W^1_p(\Omega)$ has a quasicontinuous representations defined up to a set of $p$-capacity zero. If $p>n$ then a set of $p$-capacity zero is the empty set and Sobolev functions have continuous representations.
\end{rem}

The problem of composition operators on Sobolev spaces arises in \cite{GS82} for Sobolev extension operators in cusp domains and is connected with the Reshennyak's problem (1969) \cite{VG75}.
Recall that in connection with the geometric function theory the $p$-distortion of a mapping $\varphi$ at a point $x\in\Omega$ is defined as
$$
K_p(x)=\inf \{k(x): |D\varphi(x)|\leq k(x) |J(x,\varphi)|^{\frac{1}{p}},\,\,x\in\Omega \}.
$$
If $p=n$ we have the usual conformal dilatation and in the case $p\ne n$ the $p$-dilatation arises in \cite{Ge69} (see, also, \cite{V88}). 

The geometric theory of composition operators on Sobolev spaces is based on the measure property of norms of composition operators (introduced in \cite{U93} and in the limit case $p=\infty$ in \cite{U04}).

\begin{thm}
\label{CompPhi} Let a homeomorphism $\varphi:\Omega\to\widetilde{\Omega}$
between two domains $\Omega$ and $\widetilde{\Omega}$ induces a bounded composition
operator 
\[
\varphi^{\ast}:L^1_p(\widetilde{\Omega})\to L^1_q(\Omega),\,\,\,1\leq q< p\leq\infty.
\]
Then
\[
\Phi(\widetilde{A})=\sup\limits_{f\in L^1_p(\widetilde{A})\cap C_0(\widetilde{A})}\left(\frac{\|\varphi^{\ast}(f)\mid L^1_q(\Omega)\|}{\|f\mid L^1_p(\widetilde{A})\|}\right)^{\kappa},
\]
(where $1/q-1/p=1/{\kappa}$) is a bounded monotone countably additive set function defined on open bounded subsets $\widetilde{A}\subset\widetilde{\Omega}$.
\end{thm}

The following theorem allows us to refine this function $\Phi$ as a measure generated by the $p$-distortion $K_p$.

\begin{thm}
\label{CompTh} A homeomorphism $\varphi:\Omega\to\widetilde{\Omega}$
between two domains $\Omega$ and $\widetilde{\Omega}$ induces a bounded composition
operator 
\[
\varphi^{\ast}:L^1_p(\widetilde{\Omega})\to L^1_q(\Omega),\,\,\,1\leq q\leq p\leq\infty,
\]
 if and only if $\varphi$ is a Sobolev mapping of the class $L^1_{q}(\Omega;\widetilde{\Omega})$, has finite distortion
and 
\[
K_{p,q}(\varphi;\Omega)=\|K_p \mid L_{\kappa}(\Omega)\|<\infty,
\]
where $1/q-1/p=1/{\kappa}$ ($\kappa=\infty$, if $p=q$).
\end{thm}

This theorem was proved in \cite{U93} (see also \cite{VU02}), case $p=\infty$ was considered in \cite{GU10}. Mappings that satisfy conditions of Theorem~\ref{CompTh} are called weak $(p,q)$-quasiconformal mappings \cite{GGR95,VU98} and are a natural generalization of quasiconformal mappings ($p=q=n$). In \cite{K12}, where Sobolev spaces $W^1_p(\Omega)$ were considered as spaces of locally-in\-tegr\-able functions defined up to a set of measure zero, another proof of sufficiency of conditions of Theorem~\ref{CompTh} was given. Unfortunately, methods of \cite{K12} do not allow to prove necessity conditions of Theorem~\ref{CompTh}.

\begin{rem}
The historical survey on the theory of composition operators on Sobolev spaces can be found in \cite{V12}. Unfortunately, this useful work \cite{V12} doesn't contain essential new results and contains some non-correct citations of previous original papers \cite{U93,VU98,VU02}. Let us remark also that some proofs are not complete: for example, the main result of Section 4 of  \cite{V12} was formulated for general type of mappings, but the proof was given for homeomorphisms only (and so it repeats results of the work \cite{GU10}) and even in this case contains gaps. 
\end{rem}

\subsection{Capacity estimates of composition operators}

The composition operators on Sobolev spaces allow a capacitary description. 
Recall the notion of a variational $p$-capacity \cite{GResh}.
A condenser in a domain $\Omega\subset \mathbb R^n$ is the pair $(F_0,F_1)$ of connected disjoint closed relatively to $\Omega$ sets $F_0,F_1\subset \Omega$. A continuous function $f\in L_p^1(\Omega)$ is called an admissible function for the condenser $(F_0,F_1)$,
if the set $F_i\cap \Omega$ is contained in some connected component of the set $\operatorname{Int}\{x\vert f(x)=i\}$,\ $i=0,1$. We call as the $p$-capacity of the condenser $(F_0,F_1)$ relatively to domain $\Omega$
the following quantity:
$$
{\cp}_p(F_0,F_1;\Omega)=\inf\|f\vert L_p^1(\Omega)\|^p.
$$
Here the greatest lower bond is taken over all functions admissible for the condenser $(F_0,F_1)\subset\Omega$. If the condenser has no admissible functions we put the capacity equal to infinity. 

The following capacitory charaterization of composition operators on Sobolev spaces were given in \cite{VU98,U93}.

\begin{thm}
\label{theorem:CapacityDescPP_O}
Let $1<p<\infty$.
A homeomorphism $\varphi :\Omega\to \widetilde{\Omega}$
generates a bounded composition operator
$$
\varphi^{\ast}: L^1_p(\widetilde{\Omega})\to L^1_p(\Omega)
$$
if and only if for every condenser 
$(F_0,F_1)\subset\widetilde{\Omega}$
the inequality
$$
\cp_{p}^{1/p}(\varphi^{-1}(F_0),\varphi^{-1}(F_1);\Omega)
\leq K_{p,p}(\varphi;\Omega)\cp_{p}^{1/p}(F_0,F_1;\widetilde{\Omega})
$$
holds. 
\end{thm}

\begin{thm}
\label{theorem:CapacityDescPQ_O}
Let $1<q<p<\infty$.
A homeomorphism $\varphi :\Omega\to \widetilde{\Omega}$
generates a bounded composition operator
$$
\varphi^{\ast}: L^1_p(\widetilde{\Omega})\to L^1_q(\Omega)
$$
if and only if
there exists a bounded monotone countable-additive set function
$\Phi$ defined on open subsets of $\widetilde{\Omega}$
such that for every condenser 
$(F_0,F_1)\subset \widetilde{\Omega}$
the inequality
$$
\cp_{q}^{1/q}(\varphi^{-1}(F_0),\varphi^{-1}(F_1);\Omega)
\leq\Phi(\widetilde{\Omega}\setminus(F_0\cup F_1))^{\frac{p-q}{pq}}
\cp_{p}^{1/p}(F_0,F_1;\widetilde{\Omega})
$$
holds. 
\end{thm}

In Section 2.1 we defined $p$-capacity of an arbitrary sets and discussed measurable for $p$-capacity sets. In  \cite{HKM,M} it was proved that the $p$-capacity is an outer measure. 

Now we give the capacitary distortion estimates of Borel sets under homeomorphisms generating composition operators on Sobolev spaces. 

\begin{thm}
\label{theorem:CapacityDescPP}
Let a homeomorphism $\varphi :\Omega\to \widetilde{\Omega}$
generates a bounded composition operator
$$
\varphi^{\ast}: L^1_p(\widetilde{\Omega})\to L^1_p(\Omega),\,\,1\leq p<\infty.
$$
Then the inequality
$$
\cp_{p}^{1/p}(\varphi^{-1}(\widetilde{E});\Omega)
\leq K_{p,p}(\varphi;\Omega)\cp_{p}^{1/p}(\widetilde{E};\widetilde{\Omega})
$$
holds for every Borel set $\widetilde{E}\subset\widetilde{\Omega}$. 
\end{thm}

\begin{proof}

Let $F\subset E=\varphi^{-1}(\widetilde{E})$ be a compact set. Because $\varphi$ is a homeomorphism $\widetilde{F}=\varphi(F)\subset\widetilde{E}$ is also a compact set.  Let $f\in C_0(\widetilde{\Omega})\cap L^1_p(\widetilde{\Omega)}$ be an arbitrary function such that $f\geq 1$ on $\widetilde{F}$. Then the composition $g=\varphi^{\ast}(f)$ belongs to $C_0(\Omega)\cap L^1_p({\Omega)}$, $g\geq 1$ on $F$ and 
$$
\|\varphi^{\ast}(f)\mid L^1_p(\Omega)\|\leq K_{p,p}(\varphi;\Omega) \|f \mid L^1_p(\Omega)\|.
$$
Since the function $g=\varphi^{\ast}(f)\in C_0(\Omega)\cap L^1_p({\Omega)}$ is an admissible function for the compact $F\subset E$, then
$$
\cp_{p}^{1/p}(\varphi^{-1}(\widetilde{F});\Omega)\leq \|\varphi^{\ast}(f)\mid L^1_p(\Omega)\|\leq K_{p,p}(\varphi;\Omega) \|f \mid L^1_p(\Omega)\|.
$$
Taking infimum over all functions $f\in C_0(\widetilde{\Omega})\cap L^1_p(\widetilde{\Omega)}$ such that $f\geq 1$ on $\widetilde{F}$ we have 
$$
\cp_{p}^{1/p}(\varphi^{-1}(\widetilde{F});\Omega)
\leq K_{p,p}(\varphi;\Omega)\cp_{p}^{1/p}(\widetilde{F};\widetilde{\Omega})
$$
for any compact set $\widetilde{F}\subset \widetilde{E}\subset \widetilde{\Omega}$.

Now for the Borel set $\widetilde{E}\subset \widetilde{\Omega}$ we have (by the definition of the $p$-capacity of Borel sets)
$$
\cp_{p}^{1/p}(\widetilde{F};\widetilde{\Omega})\leq \underline{\cp}_{p}^{1/p}(\widetilde{E};\widetilde{\Omega})=\cp_{p}^{1/p}(\widetilde{E};\widetilde{\Omega}).
$$
Hence 
$$
\cp_{p}^{1/p}(\varphi^{-1}(\widetilde{F});\Omega)\leq K_{p,p}(\varphi;\Omega)\cp_{p}^{1/p}(\widetilde{E};\widetilde{\Omega}).
$$
Since $F=\varphi^{-1}(\widetilde{F})$ is an arbitrary compact set, $F\subset E$, $E$ is a Borel set as a preimage of the Borel set $\widetilde{E}$ under the homeomorphism $\varphi$, then
\begin{multline*}
\cp_{p}^{1/p}(\varphi^{-1}(\widetilde{E});\Omega)=\underline{\cp}_{p}^{1/p}(\varphi^{-1}(\widetilde{E});\Omega)=
\sup\limits_{F\subset E}\cp_{p}^{1/p}(F;\Omega)\\
\leq K_{p,p}(\varphi;\Omega)\cp_{p}^{1/p}(\widetilde{E};\widetilde{\Omega}).
\end{multline*}

\end{proof}

\begin{thm}
\label{theorem:CapacityDescPQ}
Let a homeomorphism $\varphi :\Omega\to \widetilde{\Omega}$
generates a bounded composition operator
$$
\varphi^{\ast}: L^1_p(\widetilde{\Omega})\to L^1_q(\Omega),\,\,1\leq q< p<\infty.
$$
Then the inequality
$$
\cp_{q}^{1/q}(\varphi^{-1}(\widetilde{E});\Omega)
\leq K_{p,q}(\varphi;\Omega)\cp_{p}^{1/p}(\widetilde{E};\widetilde{\Omega})
$$
holds for every Borel set $\widetilde{E}\subset\widetilde{\Omega}$. 
\end{thm}

\begin{proof}

Let $F\subset E=\varphi^{-1}(\widetilde{E})$ be a compact set. Because $\varphi$ is a homeomorphism $\widetilde{F}=\varphi(F)\subset\widetilde{E}$ is also a compact set.  Let $f\in C_0(\widetilde{\Omega})\cap L^1_p(\widetilde{\Omega)}$ be an arbitrary function such that $f\geq 1$ on $\widetilde{F}$. Then the composition $g=\varphi^{\ast}(f)$ belongs to $C_0(\Omega)\cap L^1_q({\Omega)}$, $g\geq 1$ on $F$ and 
$$
\|\varphi^{\ast}(f)\mid L^1_q(\Omega)\|\leq K_{p,q}(\varphi;\Omega) \|f \mid L^1_p(\Omega)\|.
$$
Since the function $g=\varphi^{\ast}(f)\in C_0(\Omega)\cap L^1_q({\Omega)}$ is an admissible function for the compact $F\subset E$, then
$$
\cp_{q}^{1/q}(\varphi^{-1}(\widetilde{F});\Omega)\leq \|\varphi^{\ast}(f)\mid L^1_q(\Omega)\|\leq K_{p,q}(\varphi;\Omega) \|f \mid L^1_p(\Omega)\|.
$$
Taking infimum over all functions $f\in C_0(\widetilde{\Omega})\cap L^1_p(\widetilde{\Omega)}$ such that $f\geq 1$ on $\widetilde{F}$ we have 
$$
\cp_{q}^{1/q}(\varphi^{-1}(\widetilde{F});\Omega)
\leq K_{p,q}(\varphi;\Omega)\cp_{p}^{1/p}(\widetilde{F};\widetilde{\Omega})
$$
for any compact set $\widetilde{F}\subset \widetilde{E}\subset \widetilde{\Omega}$.

Now for the Borel set $\widetilde{E}\subset \widetilde{\Omega}$ we have (by the definition of the $p$-capacity of Borel sets)
$$
\cp_{p}^{1/p}(\widetilde{F};\widetilde{\Omega})\leq \underline{\cp}_{p}^{1/p}(\widetilde{E};\widetilde{\Omega})=\cp_{p}^{1/p}(\widetilde{E};\widetilde{\Omega}).
$$
Hence 
$$
\cp_{q}^{1/q}(\varphi^{-1}(\widetilde{F});\Omega)\leq K_{p,q}(\varphi;\Omega)\cp_{p}^{1/p}(\widetilde{E};\widetilde{\Omega}).
$$
Since $F=\varphi^{-1}(\widetilde{F})$ is an arbitrary compact set, $F\subset E$, $E$ is a Borel set as a preimage of the Borel set $\widetilde{E}$ under the homeomorphism $\varphi$, then
\begin{multline*}
\cp_{q}^{1/q}(\varphi^{-1}(\widetilde{E});\Omega)=\underline{\cp}_{q}^{1/q}(\varphi^{-1}(\widetilde{E});\Omega)=
\sup\limits_{F\subset E}\cp_{q}^{1/q}(F;\Omega)\\
\leq K_{p,q}(\varphi;\Omega)\cp_{p}^{1/p}(\widetilde{E};\widetilde{\Omega}).
\end{multline*}

\end{proof}

\section{Composition operators and  Ball's classes}

In this section we consider applications of the geometric theory of composition operators on Sobolev spaces to nonlinear elasticity problems.  These application build on a notion of an inner distortion which is used  for a study of "inverse" composition operators. This notion gives a geometric interpretation of the integrability condition of  $\adj D\varphi$ in the original definition of Ball's classes $A_{q,q'}^{+}(\Omega)$. 

\subsection{Inverse composition operators}

Let $\Omega$ and $\widetilde{\Omega}$ be two bounded domains in $\mathbb R^n$ and $\varphi: \Omega \to \widetilde{\Omega}$
be a mapping of finite distortion of the class $W^1_{1,\loc}(\Omega;\widetilde{\Omega})$.   We define the "normalized" inner $s$-distortion of $\varphi$ at a point $x$ as
$$
K^{I}_{s}(x,\varphi)=
\begin{cases}
\frac{|J(x,\varphi)|^{\frac{1}{s}}}{l(D\varphi(x))},\,\,&J(x,\varphi)\ne 0,\\
0,\,\,&J(x,\varphi)=0,
\end{cases}
$$
where $l(D\varphi(x))$ is defined as $\min\limits_{h=1}|D\varphi(x)\cdot h|$.

Its global integral version is called the inner $(q,s)$-distortion, $1\leq s\leq q\leq\infty$:
$$
K^{I}_{q,s}(\varphi; \Omega)=\|K^{I}_{s}(\varphi)\mid L_{\kappa}(\Omega)\|, \,\,{1}/{\kappa}={1}/{s}-{1}/{q},\,\, (\kappa=\infty,\,\, \text{if}\,\, q=s).
$$

\begin{lem}
\label{remb}
Let a homeomorphism $\varphi:\Omega\to\widetilde{\Omega}$ belongs to $L^1_q(\Omega;\widetilde{\Omega})$ and $J(x,\varphi)>0$ for almost all $x\in\Omega$.
Then 
$$  
\left(\int\limits_{\Omega}\left(\frac{|J(x,\varphi)|}{l(D\varphi(x))}\right)^{\frac{q}{q-1}}~dx\right)^{\frac{q-1}{q}}=
\left(\int\limits_{\Omega}\left|\adj D\varphi(x)\right|^{\frac{q}{q-1}}~dx\right)^{\frac{q-1}{q}}.
$$
\end{lem}

\begin{proof} Using the following equalities (see, for example, \cite{GU10}):
$$
(D\varphi(x))^{-1}=J^{-1}(x,\varphi)\adj D\varphi(x)
$$
and 
$$
\min\limits_{h=1}|D\varphi(x)\cdot h|=\left(\max\limits_{h=1}|(D\varphi(x))^{-1}\cdot h|\right)^{-1}
$$
we have 
$$
K^{I}_{q,1}(\varphi; \Omega)=\left(\int\limits_{\Omega}\left(\frac{|J(x,\varphi)|}{l(D\varphi(x))}\right)^{\frac{q}{q-1}}~dx\right)^{\frac{q-1}{q}}=
\left(\int\limits_{\Omega}\left|\adj D\varphi(x)\right|^{\frac{q}{q-1}}~dx\right)^{\frac{q-1}{q}}.
$$
\end{proof}

The following theorem give the characterization of composition operators in Sobolev spaces in the terms of the inner $(q,s)$-distortion.

\begin{thm}
\label{theorem:lowpq}
Let $\varphi:\Omega\to\widetilde{\Omega}$ be a Sobolev homeomorphism of finite distortion and belongs to $L^1_q(\Omega;\widetilde{\Omega})$.
Then the inverse homeomorphism $\varphi^{-1}:\widetilde{\Omega}\to\Omega$ generates a bounded composition operator 
\begin{equation}
\label{ineqpq}
\left(\varphi^{-1}\right)^{\ast}: L^1_q(\Omega) \to L^1_s(\widetilde{\Omega}),\,\,1\leq s<q<\infty,
\end{equation}
if and only if $\varphi^{-1}\in L^1_s(\widetilde{\Omega};\Omega)$, possesses the Luzin $N^{-1}$-property ($\varphi$ possesses the Luzin $N$-property) and 
$$
K^{I}_{q,s}(\varphi; \Omega)=\left(\int\limits_{\Omega}K^{I}_{s}(x,\varphi)^{\frac{qs}{q-s}}~dx\right)^{\frac{q-s}{qs}}<\infty.
$$
\end{thm}

\begin{proof} 
{\it Necessity.} 
Let the inverse mapping $\varphi^{-1}:\widetilde{\Omega}\to\Omega$ generates a bounded composition operator 
$$
\left(\varphi^{-1}\right)^{\ast}: L^1_q(\Omega)\to L^1_s(\widetilde{\Omega}),\,\,1\leq s<q<\infty.
$$
Then by Theorem~\ref{CompTh}, the inverse mapping $\varphi^{-1}:\widetilde{\Omega}\to\Omega$
belongs to the Sobolev space $L^1_{s}(\widetilde{\Omega};\Omega)$, has finite distortion and 
$$
\left(\int\limits_{\widetilde{\Omega}}\left(\frac{|D\varphi^{-1}(y)|^q}{|J(y,\varphi^{-1})|}\right)^{\frac{s}{q-s}}~dy\right)^{\frac{q-s}{qs}}<\infty,\,\,1\leq s< q<\infty.
$$

Note, that in the case $q\geq n$,the homeomorphism $\varphi:\Omega\to\widetilde{\Omega}$ of the class $L^1_q(\Omega;\widetilde{\Omega})$ possesses the Luzin $N$-property \cite{GResh} and in the case $1\leq s< q<n$ the mapping $\varphi^{-1}$ which generates a bounded composition operator
$$
\left(\varphi^{-1}\right)^{\ast}: L^1_q(\Omega)\to L^1_s(\widetilde{\Omega}),\,\,1\leq s< q<n,
$$
possesses the Luzin $N^{-1}$-property  \cite{VU02}. It means that $\varphi:\Omega\to\widetilde{\Omega}$ possesses the Luzin $N$-property for all $1\leq s<q<\infty$.

Hence, under the conditions of the theorem, $\varphi^{-1}\in L^1_s(\widetilde{\Omega};\Omega)$, possesses the Luzin $N^{-1}$-property for all $1\leq s<q<\infty$ and we can put  $K^{I}_{q,s}(\varphi; \Omega)=0$ on the set $Z=\{x\in\Omega: J(x,\varphi)=0\}$. So
\begin{multline*}
K^{I}_{q,s}(\varphi; \Omega)=\left(\int\limits_{\Omega\setminus Z}\left(\frac{|J(x,\varphi)|}{l(D\varphi(x))^s}\right)^{\frac{q}{q-s}}~dx\right)^{\frac{q-s}{qs}}\\
=\left(\int\limits_{\Omega\setminus Z}\left(\frac{|D\varphi^{-1}(\varphi(x))|^q}{|J(\varphi(x),\varphi^{-1})|}\right)^{\frac{s}{q-s}}|J(x,\varphi)|~dx\right)^{\frac{q-s}{qs}}\\
=\left(\int\limits_{\widetilde{\Omega}}\left(\frac{|D\varphi^{-1}(y)|^q}{|J(y,\varphi^{-1})|}\right)^{\frac{s}{q-s}}~dy\right)^{\frac{q-s}{qs}}<\infty.
\end{multline*}

{\it Sufficiency.} Let $\varphi^{-1}\in L^1_q(\widetilde{\Omega};\Omega)$, possesses the Luzin $N^{-1}$-property and 
$$
K^{I}_{q,s}(\varphi; \Omega)=\left(\int\limits_{\Omega}\left(\frac{|J(x,\varphi)|}{l(D\varphi(x))^s}\right)^{\frac{q}{q-s}}~dx\right)^{\frac{q-s}{qs}}<\infty.
$$
Then
\begin{multline*}
\left(\int\limits_{\widetilde{\Omega}}\left(\frac{|D\varphi^{-1}(y)|^q}{|J(y,\varphi^{-1})|}\right)^{\frac{s}{q-s}}~dy\right)^{\frac{q-s}{qs}}=\left(\int\limits_{\Omega\setminus Z}\left(\frac{|D\varphi^{-1}(\varphi(x))|^q}{|J(\varphi(x),\varphi^{-1})|}\right)^{\frac{s}{q-s}}|J(x,\varphi)|~dx\right)^{\frac{q-s}{qs}}\\
=\left(\int\limits_{\Omega\setminus Z}\left(\frac{|J(x,\varphi)|}{l(D\varphi(x))^s}\right)^{\frac{q}{q-s}}~dx\right)^{\frac{q-s}{qs}}=K^{I}_{q,s}(\varphi; \Omega)<\infty.
\end{multline*}
Then by Theorem~\ref{CompTh} the mapping $\varphi^{-1}$ generates a bounded composition operator 
$$
\left(\varphi^{-1}\right)^{\ast}: L^1_q(\Omega)\to L^1_s(\widetilde{\Omega}),\,\,1\leq s<q<\infty.
$$
\end{proof}

Now we give the description of Ball's classes $A^{+}_{q,q'}(\Omega;\widetilde{\Omega})$ in the terms of composition operators on Sobolev spaces.

\begin{thm}
\label{theorem:ballcomp}
The homeomorphism $\varphi:\Omega\to\widetilde{\Omega}$ between bounded domains $\Omega,\widetilde{\Omega}\subset\mathbb R^n$ belongs to the Ball class $A^{+}_{q,q'}(\Omega;\widetilde{\Omega})$ for $q>n-1$, $1/q+1/q'=1$, if and only if $\varphi\in L^1_q(\Omega;\widetilde{\Omega})$, possesses the Luzin $N$-property and the inverse mapping generates the bounded composition operator
$$
\left(\varphi^{-1}\right)^{\ast}: L^1_q(\Omega)\to L^1_1(\widetilde{\Omega}),
$$
with $\|\left(\varphi^{-1}\right)^{\ast}\|\leq \|\adj D\varphi \mid L_{q'}(\Omega)\|$.
\end{thm}

\begin{proof}
By Lemma~\ref{remb} we have, that the inner integral distortion 
$$
K^{I}_{q,1}(\varphi; \Omega)=\left(\int\limits_{\Omega}\left(\frac{|J(x,\varphi)|}{l(D\varphi(x))}\right)^{\frac{q}{q-1}}~dx\right)^{\frac{q-1}{q}}=
\left(\int\limits_{\Omega}\left|\adj D\varphi(x)\right|^{\frac{q}{q-1}}~dx\right)^{\frac{q-1}{q}}<\infty
$$
Hence, by Theorem~\ref{theorem:lowpq} the inverse mapping generates the bounded composition operator
$$
\left(\varphi^{-1}\right)^{\ast}: L^1_q(\Omega)\to L^1_1(\widetilde{\Omega}),
$$ 
if and only if $\varphi$ belongs to the Ball class $A^{+}_{q,q'}(\Omega;\widetilde{\Omega})$, $q>n-1$, $1/q+1/q'=1$. The estimate of the composition operator norm follows from \cite{U93,VU02}.

\end{proof}

Now we consider two-sides estimates of composition operators.
In \cite{GU10} it was proved that the inequality 
$$
\|\varphi^{\ast}(f)\mid L^1_q(\Omega)\|
\leq \|\varphi\mid L^1_q(\Omega)\|\cdot\|f\mid L^1_{\infty}(\widetilde{\Omega})\|,\,\,1\leq q<\infty,
$$
holds for any $f\in L^1_{\infty}(\widetilde{\Omega})$ if and only if $\varphi\in L^1_q(\Omega)$. 

Combining the previous theorem and this inequality we immediately obtain

\begin{thm} Let the homeomorphism $\varphi:\Omega\to\widetilde{\Omega}$ between bounded domains $\Omega,\widetilde{\Omega}\subset\mathbb R^n$ belongs to the Ball class $A^{+}_{q,q'}(\Omega;\widetilde{\Omega})$ for $q>n-1$, $1/q+1/q'=1$
Then for any function $f\in L^1_{\infty}(\widetilde{\Omega})$ following inequalities
\begin{equation*}
%\label{ineq1}
{\|\adj D\varphi \mid L_{q'}(\Omega)\|}^{-1}\cdot\|f\mid L^1_1(\widetilde{\Omega})\|\leq \|\varphi^{\ast}(f)\mid L^1_q(\Omega)\|\leq \|\varphi\mid L^1_q(\Omega)\|\cdot\|f\mid L^1_{\infty}(\widetilde{\Omega})\|
\end{equation*}
hold.
\end{thm}

This theorem demonstrates variations of the nonlinear elastic potential energy under weak quasiconformal deformations of elastic bodies.

\subsection{The weak regularity of inverse Sobolev mappings}

In the geometric theory of composition operators on Sobolev spaces the significant role plays the following composition duality property \cite{U93}. This property represents the weak inverse mapping theorem (in the part of regularity of inverse mappings) for Sobolev mappings.

\begin{thm}
\label{CompDual} Let a homeomorphism $\varphi:\Omega\to\widetilde{\Omega}$,  $\Omega,\widetilde{\Omega}\subset\mathbb R^n$, induces a bounded composition
operator 
\[
\varphi^{\ast}:L^1_p(\widetilde{\Omega})\to L^1_q(\Omega),\,\,\,n-1<q\leq p<\infty.
\]
Then the inverse mapping $\varphi^{-1}:\widetilde{\Omega}\to\Omega$ induces a bounded composition operator 
\[
\left(\varphi^{-1}\right)^{\ast}:L^1_{\tilde{q}}(\Omega)\to L^1_{\tilde{p}}(\widetilde{\Omega}),\,\,\,n-1<\tilde{p}\leq \tilde{q}<\infty,
\]
where $\tilde{p}=p/(p-n+1)$ and $\tilde{q}=q/(q-n+1)$.
\end{thm}

In the present work we prove this property in the limit case $p=\infty$.

\begin{thm}
\label{CompDualLim} Let a homeomorphism $\varphi:\Omega\to\widetilde{\Omega}$,  $\Omega,\widetilde{\Omega}\subset\mathbb R^n$, be a mappings of finite distortion, possesses the Luzin $N$-property and induces a bounded composition
operator 
\[
\varphi^{\ast}:L^1_{\infty}(\widetilde{\Omega})\to L^1_q(\Omega),\,\,\,n-1<q<\infty.
\]
Then the inverse mapping $\varphi^{-1}:\widetilde{\Omega}\to\Omega$ induces a bounded composition operator 
\[
\left(\varphi^{-1}\right)^{\ast}:L^1_{\tilde{q}}(\Omega)\to L^1_{1}(\widetilde{\Omega}),\,\,\,\tilde{q}=q/(q-n+1).
\]
\end{thm}

\begin{proof}
Since $\varphi:\Omega\to\widetilde{\Omega}$ generates a bounded composition
operator
\[
\varphi^{\ast}:L^1_{\infty}(\widetilde{\Omega})\to
L^1_{q}(\Omega),\,\,\,n-1<q< \infty,
\]
then by \cite{GU10} the mapping $\varphi\in L^1_{q}(\Omega;\widetilde{\Omega})$. Because $\varphi$ possesses the Luzin $N$-property, then the inverse mapping belongs to $W^1_{1,\loc}(\widetilde{\Omega})$ and is a mapping of finite distortion \cite{GU10}. Denote by $Z=\{x\in\Omega \mid J(x,\varphi)=0\}$ and $S$ is the set from the change of variables formula (\ref{chvf}), $|S|=0$. Then \cite{U93}
\[
|D\varphi^{-1}(y)|\leq\frac{|D\varphi(x)|^{n-1}}{|J(x,\varphi)|},
\]
for almost all $x\in \Omega\setminus \left(S\cup Z\right)$, $y=\varphi(x)\in \widetilde{\Omega}\setminus \varphi\left(S\cup Z\right)$,
and  
$$
|D\varphi^{-1}(y)|=0\,\,\text{for almost all}\,\,y\in \varphi(S).
$$
Then
\begin{multline*}
\int\limits_{\widetilde{\Omega}}\left(\frac{|D\varphi^{-1}(y)|^{\tilde{q}}}{|J(y,\varphi^{-1})|}\right)^{\frac{1}{\tilde{q}-1}}~dy=\int\limits_{\widetilde{\Omega}\setminus \varphi\left(S\cup Z\right)}\left(\frac{|D\varphi^{-1}(y)|^{\tilde{q}}}{|J(y,\varphi^{-1})|}\right)^{\frac{1}{\tilde{q}-1}}~dy\\
\leq
\int\limits_{\widetilde{\Omega}\setminus \varphi\left(S\cup Z\right)}\left(\left(\frac{|D\varphi(\varphi^{-1}(y))|^{n-1}}{|J(\varphi^{-1}(y),\varphi)|}\right)^{\tilde{q}}\cdot\frac{1}{|J(y,\varphi^{-1})|}\right)^{\frac{1}{\tilde{q}-1}}~dy\\
=
\int\limits_{\widetilde{\Omega}\setminus \varphi\left(S\cup Z\right)}\frac{|D\varphi(\varphi^{-1}(y))|^{q}}{|J(\varphi^{-1}(y),\varphi)|}~dy
=\int\limits_{\Omega\setminus \left(S\cup Z\right)}\frac{|D\varphi(x)|^{q}}{|J(x,\varphi)|}|J(x,\varphi)|~dx\\
\leq
\int\limits_{\Omega}|D\varphi(x)|^{q}~dx<\infty.
\end{multline*}
Hence \cite{U93} $\varphi^{-1}:\widetilde{\Omega}\to\Omega$ generates a bounded composition operator
\[
\left(\varphi^{-1}\right)^{\ast}:L^1_{\tilde{q}}(\Omega)\to L^1_{1}(\widetilde{\Omega}),
\]
where $\tilde{q}=q/(q-n+1)$. 
\end{proof}

\begin{cor}
Let $\varphi:\Omega\to\widetilde{\Omega}$ be a Sobolev homeomorphism of bounded domains $\Omega,\widetilde{\Omega}$ such that 
$J(x,\varphi)>0$ for almost all $x\in\Omega$. Then $\varphi\in L^1_n(\Omega;\widetilde{\Omega})$ if and only if 
$\varphi\in A^{+}_{n,n'}(\Omega;\widetilde{\Omega})$, $n'=n/(n-1)$.
\end{cor}

\begin{proof}
The inclusion 
$$
A^{+}_{n,n'}(\Omega;\widetilde{\Omega})\subset L^1_n(\Omega;\widetilde{\Omega})
$$
holds by the definition of Ball's class $A^{+}_{n,n'}(\Omega;\widetilde{\Omega})$. Now let $\varphi\in L^1_n(\Omega;\widetilde{\Omega})$. Then by Theorem~\ref{CompDualLim} the inverse mapping $\varphi^{-1}:\widetilde{\Omega}\to\Omega$ induces a bounded composition operator 
$$
\left(\varphi^{-1}\right)^{\ast}:L^1_{n}(\Omega)\to L^1_{1}(\widetilde{\Omega}).
$$
By Theorem~\ref{theorem:ballcomp} the mapping $\varphi\in A^{+}_{n,n'}(\Omega;\widetilde{\Omega})$, $n'=n/(n-1)$.
\end{proof}

The key point in proof of the regularity of mappings of Ball's classes plays the regularity of mappings which are inverse to Sobolev mappings. This topic arises in \cite{Z69} and was studied by many authors, see, for example, \cite{CHM,GU10,HKM06,HKO07,U04}. In the present work we use the following theorem from \cite{GU10}:

\begin{thm} \cite{GU10} Let a homeomorphism $\varphi:\Omega\to \widetilde{\Omega}$ between two domains $\Omega,\widetilde{\Omega}\subset\mathbb R^n$ belong to the Sobolev space $L^1_{n-1}(\Omega;\widetilde{\Omega})$, possess the Luzin $N$-property and have finite distortion. Then the inverse mapping $\varphi^{-1}$ belongs to the Sobolev space $L^1_1(\widetilde{\Omega};\Omega)$.
\end{thm}

In \cite{CHM} was proved the local version of this theorem, without the assumption of the Luzin $N$-property. But it seems that the work \cite{CHM} has the following gap: 

{\it Let  $\varphi:\Omega\to \widetilde{\Omega}$ be a homeomorphism between two domains $\Omega,\widetilde{\Omega}\subset\mathbb R^n$ belong to the Sobolev space $L^1_{n-1}(\Omega;\widetilde{\Omega})$. Suppose that $|J(x,\varphi)|\ne 0$ for almost all $x\in\Omega$. Then in \cite{CHM} is defined a "weak inverse upper gradient" $g$ by the formula
$$
g(\varphi(x))=
\begin{cases}
\frac{|\adj D\varphi(x)|}{|J(x,\varphi)|},\,\,&\text{if}\,\,x\in \Omega\setminus S,\\
0, \,\,&\text{if}\,\,x\in S,
\end{cases}
$$
where $S$ is the set from the change of variables formula (\ref{chvf}), $|S|=0$. So, if we define a "weak inverse upper gradient" $g$ by the formula
$$
\widetilde{g}(\varphi(x))=
\begin{cases}
\frac{|\adj D\varphi(x)|}{|J(x,\varphi)|},\,\,&\text{if}\,\,x\in \Omega\setminus S,\\
1, \,\,&\text{if}\,\,x\in S,
\end{cases}
$$
it will be the same measurable function, but the second inequality on the page 233 of \cite{CHM} does not hold. 

}

\section{Measure and capacity distortion estimates}

\subsection{Measure distortion estimates}

In this section we give the volume distortion property of mappings of Ball's classes. Let us recall the following theorem \cite{VU02} in the convenient for us form, because in \cite{VU02} this theorem was proved for general (not necessary homeomorphic mappings).
Let $\varphi:\Omega\to\widetilde{\Omega}$ be a Sobolev mapping, then the inverse $s$-distortion function is defined \cite{VU02} by
$$
H_s(y)=
\begin{cases}
\left(\sum\limits_{x\in\varphi^{-1}(y)\setminus S, J(x,\varphi)\ne 0} \frac{|D\varphi(x)|^s}{|J(x,\varphi)|} \right)^{\frac{1}{s}},\\
0, & \text{otherwise},
\end{cases}
$$
where $S$ is the set from the change of variables formula (\ref{chvf}).

\begin{thm}\cite{VU02}
\label{distqs}
Let a homeomorphism $\varphi:\Omega\to\widetilde{\Omega}$ generates a bounded composition operator
$$
\varphi^{\ast}: L^1_q(\widetilde{\Omega})\to L^1_s(\Omega),\,\,1\leq s\leq q\leq n.
$$
Then for any measurable set $\widetilde{A}\subset\widetilde{\Omega}$ the following inequality
\begin{equation}
\label{mes}
|\varphi^{-1}(\widetilde{A})|^{\frac{1}{s}-\frac{1}{n}}\leq \|H_s \mid L_{\kappa}(\widetilde{A})\| |\widetilde{A}|^{\frac{1}{q}-\frac{1}{n}},\,\,{1}/{\kappa}={1}/{s}-{1}/{q},
\end{equation}
holds.
\end{thm}

Hence, we obtain

\begin{thm}
\label{distmes}
Let a homeomorphism $\varphi$ belongs to the Ball's class $A^{+}_{q,q'}(\Omega;\widetilde{\Omega})$, $n-1< q\leq n$, $1/q+1/q'=1$, then the inequality
$$
|\varphi(A)|^{1-\frac{1}{n}}\leq \left(\frac{1}{|A|}\int\limits_A|\adj D\varphi(x)|^{\frac{q}{q-1}}~dx\right)^{1-\frac{1}{q}} 
|A|^{1-\frac{1}{n}}, \quad n-1< q\leq n,
$$
holds for any measurable set $A\subset\Omega$.
\end{thm}

\begin{proof}
By Theorem~\ref{theorem:ballcomp} the inverse mapping $\varphi^{-1}:\widetilde{\Omega}\to \Omega$ generates the bounded composition operator
$$
\left(\varphi^{-1}\right)^{\ast}: L^1_q(\Omega)\to L^1_1(\widetilde{\Omega}).
$$
Hence by \cite{VU02} we have
\begin{equation*}
|\varphi(A)|^{1-\frac{1}{n}}\leq \|H_1 \mid L_{\kappa}(A)\| |A|^{\frac{1}{q}-\frac{1}{n}},\,\,{1}/{\kappa}=1-{1}/{q},
\end{equation*}

Now we calculate the norm $\|H_1 \mid L_{\kappa}(A)\|$. Since $J(x,\varphi)>0$ for almost all $x\in\Omega$ and $\varphi$ possesses the Luzin $N$-property,
then
\begin{multline*}
\|H_1 \mid L_{\kappa}(A)=\left(\int\limits_A \left(\frac{|D\varphi^{-1}(\varphi(x))|}{|J(\varphi(x),\varphi^{-1})|}\right)^{\frac{q-1}{q}}~dx\right)^{\frac{q}{q-1}}\\
=\left(\int\limits_{\Omega}\left(\frac{|J(x,\varphi)|}{l(D\varphi(x))}\right)^{\frac{q}{q-1}}~dx\right)^{\frac{q-1}{q}}=
\left(\int\limits_{\Omega}\left|\adj D\varphi(x)\right|^{\frac{q}{q-1}}~dx\right)^{\frac{q-1}{q}}.
\end{multline*}

\end{proof}

\subsection{Capacity distortion estimates}

In this section we prove that topological mappings (homeomorphism) of Ball's classes which possess the Luzin $N$-property are absolutely continuous with respect to the corresponding $p$-capacities, which considered as outer measures associated with Sobolev spaces. It refines corresponding results of \cite{HK04,Sv88}. Recall that Borel sets are measurable for the $p$-capacity \cite{M}.

In the following theorem we give the capacitary distortion estimates of the mappings with integrable inner distortion. 

\begin{thm}
\label{theorem:cap}
Let a homeomorphism of finite distortion $\varphi:\Omega\to\widetilde{\Omega}$ belong to $L^1_q(\Omega)$, $1<q<\infty$, possess the Luzin $N$-property and such that
$$
K^{I}_{q,s}(\varphi;\Omega)=\|K^{I}_{s}(\varphi)\mid L_{\kappa}(\Omega)\|<\infty,  \,\,1\leq s < q<\infty,
$$
where ${1}/{\kappa}={1}/{s}-{1}/{q}$.
Then for every Borel set $E\subset{\Omega}$
the inequality
$$
\cp_{s}^{\frac{1}{s}}(\varphi(E);\widetilde{\Omega})
\leq K^{I}_{q,s}(\varphi; \Omega)\cp_{q}^{\frac{1}{q}}(E;{\Omega})
$$
holds. 
\end{thm}

\begin{proof}
By Theorem~\ref{theorem:lowpq} the inverse mapping $\varphi^{-1}: \widetilde{\Omega}\to\Omega$ generates a bounded composition operator
$$
(\varphi^{-1})^{\ast}: L^1_q(\Omega)\to L^1_s(\widetilde{\Omega}), 1\leq s< q<\infty.
$$
Hence, by Theorem~\ref{theorem:CapacityDescPQ} for any Borel set $E\subset{\Omega}$
the inequality
$$
\cp_{s}^{\frac{1}{s}}(\varphi(E);\widetilde{\Omega})
\leq K^{I}_{q,s}(\varphi; \Omega)\cp_{q}^{\frac{1}{q}}(E;{\Omega})
$$
holds. 
\end{proof}

Using this theorem we obtain that topological mappings of Ball's classes  are absolutely continuous with respect to capacity, which is considered as an outer measure associated with the Sobolev spaces. This result refines results of \cite{HK04,Sv88}.  

\begin{thm}
\label{distcap}
Let a homeomorphism $\varphi\in A^{+}_{q,q'}(\Omega;\widetilde{\Omega})$, $1/q+1/q'=1$, then the inequality
$$
\cp_{1}(\varphi(E);\widetilde{\Omega})
\leq \left(\int\limits_\Omega|\adj D\varphi(x)|^{\frac{q}{q-1}}~dx\right)^{1-\frac{1}{q}}\cp_{q}^{\frac{1}{q}}(E;{\Omega})
$$
holds for any Borel set $E\subset\Omega$.
\end{thm}

\begin{proof}
By Lemma~\ref{remb} we have, that the inner integral distortion 
$$
K^{I}_{q,1}(\varphi; \Omega)=\left(\int\limits_{\Omega}\left(\frac{|J(x,\varphi)|}{l(D\varphi(x))}\right)^{\frac{q}{q-1}}~dx\right)^{\frac{q-1}{q}}=
\left(\int\limits_{\Omega}\left|\adj D\varphi(x)\right|^{\frac{q}{q-1}}~dx\right)^{\frac{q-1}{q}}<\infty.
$$
Hence by Theorem~\ref{theorem:cap} we obtain the required inequality.
\end{proof}

Therefore we obtain that mappings of Ball's classes  are absolutely continuous with respect to capacity.

\vskip 0.3cm

{\bf Acknowledgments}:

The first author was supported by the United States-Israel Binational Science Foundation (BSF Grant No. 2014055).

\vskip 0.3cm

Vladimir Gol'dshtein; Department of Mathematics, Ben-Gurion University of the Negev, P.O.Box 653, Beer Sheva, 8410501, Israel 
 
\emph{E-mail address:} \email{vladimir@math.bgu.ac.il} \\

Alexander Ukhlov; Department of Mathematics, Ben-Gurion University of the Negev, P.O.Box 653, Beer Sheva, 8410501, Israel 
							
\emph{E-mail address:} \email{ukhlov@math.bgu.ac.il    

\end{document}